\documentclass[12pt,reqno]{amsart}
\usepackage{amssymb}
\usepackage{amscd}
\usepackage{hyperref}

\newtheorem{theorem}{Theorem}[section]

\newtheorem{lemma}[theorem]{Lemma}
\newtheorem{proposition}[theorem]{Proposition}

\newtheorem{remark}[theorem]{Remark}

\numberwithin{equation}{section}
\begin{document}
 \title{On Riemannian foliations admitting  transversal conformal fields} 
 \author[W.~C.~Kim]{Woo Cheol  Kim}
 \address{Department of Mahtematics\\ 
 Jeju National University\\
 Jeju 690-756\\
 Republic of Korea}
 \email[W.~C.~Kim]{curi1981@naver.com}
 \author[S.~D.~Jung]{Seoung Dal Jung}
 \address{Department of Mathematics\\
 Department of Mathematics\\
 Jeju National University\\
 Jeju 690-756\\
 Republic of Korea}
 \email[S.~D.~Jung]{sdjung@jejunu.ac.kr}

\begin{abstract}  
Let $(M,g_M,\mathcal F)$ be a closed, connected Riemannian manifold with a  Riemannian foliation $\mathcal F$  of nonzero constant transversal scalar curvature. When $M$ admits a  transversal nonisometric  conformal field, we find some generalized conditions that $\mathcal F$ is transversally isometric to the sphere.  
\end{abstract}
\maketitle

\renewcommand{\thefootnote}{} \footnote{%
This paper was supported by the 2018 scientific promotion program funded by Jeju National University. 
.} \renewcommand{\thefootnote}{\arabic{footnote}} %
\setcounter{footnote}{0}

\section{Introduction}
A Riemannian foliation is a foliation $\mathcal F$ on a smooth manifold $M$ such that the normal bundle $Q=TM/T\mathcal F$ may be endowed with a metric $g_Q$  whose Lie derivative is zero along leaf directions \cite{TO}.  Note that we can choose a Riemannian metric $g_M$ on $M$ such that $g_M|_{T\mathcal F^\perp} =g_Q$; such a metric is called {\it bundle-like}.   A Riemannian foliation  $\mathcal F$ is {\it transversally isometric} to $(W,G)$, where $G$ is a discrete group acting by isometries on a Riemannian manifold $(W,g_W)$, if there exists a homeomorphism $\eta:W/G \to M/\mathcal F$ that is locally covered by isometries \cite{LR}.
 Recently,  S. D. Jung and K. Richardson \cite{JLK} proved the {\it generalized Obata theorem} which states that:
 $\mathcal F$ is  transversally isometric to a sphere $(S^q(1/c),G)$, where $G$ is the discrete subgroup of $O(q)$ acting by isometries on the last $q$ coordinates of the sphere $S^q(1/c)$ of radius $1/c$  if and only if there exists a non-constant basic function $f$ such that 
 \begin{align}
 \nabla_X \nabla f = -c^2 f X
 \end{align}
for all  foliated normal vectors $X$, where $c$ is a positive real number and  $\nabla$ is the transverse Levi-Civita connection on the normal bundle $Q$.

A {\it transversal conformal field} is a normal  vector field with a flow preserving the conformal class of the transverse metric.  That is,  the infinitesimal automorphism $Y$ is  transversal conformal  if  $L_Yg_Q=2f_Y g_Q$ for
a basic function $f_Y$ depending on $Y$ (\cite{JJ}, \cite{PY}, \cite{PY1}), where $L_Y$ is the Lie derivative.   In this case,  it is trivial that 
\begin{align}
f_Y = {1\over q}{\rm div}_\nabla(\pi (Y)),
\end{align}
 where ${\rm div}_\nabla$ is a transversal divergence and $\pi:TM\to Q$ is the natural projection.  If the transversal conformal field $Y$ satisfies ${\rm div}_\nabla(\pi (Y))=0$, i.e, $L_Y g_Q=0$,   then $Y$ is said to be  {\it transversal Killing field}, that is, its flow is a transversal infinitesimal isometry.

In this article, we study the Riemannian foliation admitting  a transversal nonisometric conformal field.  
First, we recall the well-known theorems about the Riemannian foliations  admitting a transversal nonisometric conformal field (\cite{JU4}, \cite{JJ}, \cite{JL},  \cite{JLK}). 

Let $R^Q$, ${\rm Ric}^Q$ and $\sigma^Q$ be the transversal curvature tensor, transversal Ricci operator and transversal scalar curvature  with respect to the transversal Levi-Civita connection $\nabla$ on $Q$  \cite{TO}.
Let $\kappa_B$ be the basic part of the mean curvature form $\kappa$ of the foliation $\mathcal F$ and $\kappa_B^\sharp$ its dual vector field (precisely, see Section 2).  Then we have the following well-known theorems.

\bigskip
\noindent
{\bf Theorem A. } \cite {JLK}  {\it Let $(M,g_M,\mathcal F)$ be a closed, connected Riemannian manifold with a Riemannian foliation $\mathcal F$  of a nonzero constant  transversal scalar curvature $\sigma^Q$.  If $M$ admits a transversal nonisometric conformal field $ Y$   satisfying  one of the following conditions:

 $(1)$ $Y=\nabla h$ for any basic function $h$, or
 
 $(2)$ $L_Y{\rm Ric}^Q =\mu g_Q$ for some basic function $\mu$, or 
 
  $(3)$ ${\rm Ric}^Q(\nabla f_Y)={\sigma^Q\over q}\nabla f_Y$, $g_Q(\kappa_B^\sharp,\nabla f_Y)=0$ and $g_Q(A_{\kappa_B^\sharp}\nabla f_Y,\nabla f_Y)\leq 0$, 
    
\noindent   then $\mathcal F$ is transversally isometric to the sphere $(S^q(1/c),G)$. }
  
\bigskip
\noindent{\bf Theorem B.}  \cite{JL}  {\it Let $(M,g_M,\mathcal F)$ be as in Theorem A. If $M$ admits a transversal nonisometric conformal field $Y$,  then
\begin{align*}
\int_M g_Q( {\rm Ric}^Q(\nabla f_Y), \nabla f_Y) \leq { (\sigma^Q)^2\over q(q-1)}\int_M f_Y^2.
\end{align*}
Equality holds if and only if $\mathcal F$ is transversally isometric to the sphere $(S^q(1/c),G)$.}

\begin{remark} Theorem B has been proved in  \cite{YA2} for the point foliation.
\end{remark}

\bigskip
\noindent{\bf Theorem C.}  \cite{JL} {\it Let $(M,g_M,\mathcal F)$ be a complete Riemannian manifold with a Riemannian foliation $\mathcal F$ of a positive constant transversal scalar curvature $\sigma^Q$.  If $M$ admits a transversal nonisometric conformal field $Y$, then
\begin{align*}
|\nabla \nabla f_Y |^2 \geq {(\sigma^Q)^2 \over q (q-1)^2 } f_Y^2.
\end{align*}
Equality holds if and only if  $\mathcal F$ is transversally isometric to the sphere $(S^q(1/c),G)$.}

\begin{remark}
Theorem C has been proved in \cite{GO} for the point foliation.
\end{remark}
 Now, we recall two tensor fields $E^Q$ and $Z^Q$ (\cite{JU4}, \cite{JL}) by
\begin{align}
&E^Q(Y) = {\rm Ric}^Q(Y) -{\sigma^Q \over q} Y,\quad Y\in T\mathcal F^\perp,\\
& Z^Q (X,Y) = R^Q (X,Y) - R_\sigma^Q (X,Y),
\end{align}
where $R^Q_\sigma(X,Y)s = {\sigma^Q\over q(q-1)}\{g_Q(\pi(Y),s)\pi(X) -g_Q(\pi(X),s)\pi(Y)\}$ for any vector field $X,Y \in TM$ and $s\in \Gamma Q$.  Trivially,  if $E^Q =0$  (resp. $Z^Q=0$), then  the foliation is transversally Einsteinian (resp. transversally constant sectional curvature).   The tensor $Z^Q$ is called as the transversal concircular curvature tensor, which is a generalization of the concircular curvature tensor on a Riemannian manifold.  In an ordinary manifold, the concircular curvature tensor is invariant under a concircular transformation which is a conformal transformation preserving geodesic circles \cite{YA-1}.
Then we have the well-known theorem.

\bigskip
\noindent {\bf Theorem D.}  \cite{JU4}  {\it Let $(M,g_M,\mathcal F)$ be as in Theorem A. If $M$ admits a transversal nonisometric conformal field $Y$ such that
\begin{align*}
 \int_M g_Q (E^Q (\nabla f_Y), \nabla f_Y) \geq 0,
\end{align*} 
 then $\mathcal F$ is transversally isometric to the sphere $(S^q(1/c),G)$.}

\begin{remark}   Since ${\rm Ric^Q}(\nabla f_Y) = {\sigma^Q\over q}\nabla f_Y$ implies $E^Q(\nabla f_Y)=0$, Theorem D is a generalization of Theorem A (3) when $\mathcal F$ is minimal.
\end{remark}

\bigskip
\noindent
 {\bf Theorem E.}  (\cite{JJ}, \cite{JL}) {\it Let  $(M,g_M,\mathcal F)$ be as in Theorem A, and suppose that $\mathcal F$ is minimal.  If $M$ admits a transversal nonisometric conformal field $Y$   such that
\begin{align*}
(i)\ L_Y|E^Q|^2& = 0\  \textrm{(\cite{JJ})} 
\end{align*}
or
\begin{align*}
(ii) \ L_Y|Z^Q |^2& = 0  \  \textrm{(\cite{JL})},
\end{align*}
then $\mathcal F$ is transversally isometric to the sphere $(S^q(1/c),G)$.}
\begin{remark} Theorem D and Theorem E have been proved in \cite{YA} for  the point foliation, that is, an ordinary manifold.
\end{remark}

\bigskip
\noindent In this paper, we prove the following theorems.

\bigskip
\noindent{\bf Theorem 1.} {\it Let $(M,g_M,\mathcal F)$ be as in Theorem A, and suppose that $\mathcal F$ is minimal.  If $M$ admits a transversal nonisometric conformal field $Y$ such that
\begin{align*}
L_Y |E^Q|^2 = const. \quad{\rm or} \ L_Y |Z^Q|^2=const.,
\end{align*}
then $\mathcal F$ is transversally isometric to the sphere $(S^q(1/c),G)$.
}
\begin{remark} Theorem 1 is a generalization of Theorem E.
\end{remark}
\bigskip
\noindent {\bf Theorem 2.}  {\it Let $(M,g_M,\mathcal F)$ be as in Theorem A, and suppose that $\mathcal F$ is minimal.  If $M$ admits a transversal nonisometric conformal field $Y$ such that
\begin{align*}
L_Y g_Q(L_YE^Q,E^Q) \leq 0,
\end{align*}
then $\mathcal F$ is transversally isometric to the sphere $(S^q(1/c),G)$.}

\begin{remark}
 Theorem 2 is a generalization of Theorem A (2) and (3) when $\mathcal F$ is minimal (cf. Remark. 4.3).
\end{remark}

\bigskip
\noindent{\bf Theorem 3.} {\it  Let $(M,g_M,\mathcal F)$ be as in Theorem A.  If $M$ admits a transversal conformal field $Y$ such that $Y= K + \nabla h$, where $K$ is a transversal Killing field and $h$ is a basic function, then 
$\mathcal F$ is transversally isometric to the sphere $(S^q(1/c),G)$}.

\begin{remark}
Theorem 3 is a generalization of Theorem A (1).
\end{remark}

\section{preliminaries}
  
 Let $(M,g_M,\mathcal F)$ be a $(p+q)$-dimensional
Riemannian manifold with a foliation $\mathcal F$ of codimension
$q$ and a bundle-like metric $g_M$ with respect to $\mathcal F$
[16]. Let $TM$ be the tangent bundle of $M$, $T\mathcal F$ its integrable
subbundle given by $\mathcal F$, and  $Q=TM/T\mathcal F$ the corresponding
normal bundle. Then there exists an exact sequence of vector
bundles
\begin{equation*}
 0 \longrightarrow T\mathcal F \longrightarrow
TM_{\buildrel \longleftarrow \over \sigma }^{\buildrel \pi \over \longrightarrow} Q \longrightarrow
0,
\end{equation*}
where $\pi:TM\to Q$ is a natural projection and $\sigma:Q\to T\mathcal F^\perp$ is a bundle map satisfying
$\pi\circ\sigma={\rm id}$.  Let $g_Q$ be the holonomy invariant metric
on $Q$ induced by $g_M$, that is,  $L_X g_Q=0$ for any $X\in T\mathcal F$, where
$L_X$ is the transversal Lie derivative, which is defined by 
$L_X s=\pi[X,\sigma(s)]$ for any $s\in \Gamma Q$. Let $\nabla$ be the transverse Levi-Civita
connection in $Q$  \cite{KT1}.
    The transversal curvature tensor $R^Q$ of $\nabla$ is defined by $
R^Q(X,Y)=[\nabla_X,\nabla_Y]-\nabla_{[X,Y]}$  for any vector fields
$X,Y\in\Gamma TM$. Let ${\rm Ric}^Q$ and $\sigma^Q$ be the transversal Ricci operator and the transversal scalar curvature of $\mathcal F$, respectively.
 The foliation $\mathcal F$ is said to be (transversally) {\it Einsteinian}  if ${\rm Ric}^Q=\frac1q \sigma^Q\cdot {\rm id}$ with constant transversal scalar curvature  $\sigma^Q.$
The mean curvature vector field $\tau$ is defined by
\begin{equation}
\tau=\sum_{i=1}^p\pi(\nabla^M_{f_i}f_i),
\end{equation}
 where $\{f_i\}(i=1,\cdots,p)$ is a local orthonormal frame field on $T\mathcal F$. 
  The foliation $\mathcal F$ is said to be {\it
minimal} if the mean curvature vector field  $\tau$ vanishes. 
Let $\{e_a\} (a=1,\cdots,q)$ be a local orthonormal frame field on $Q$. For any $s\in \Gamma Q$, the transversal divergence ${\rm div}_\nabla(s)$  is given by
\begin{align}
{\rm div}_\nabla(s) =\sum_{a=1}^q g_Q(\nabla_{e_a}s,e_a).
\end{align}
For the later use, we recall the transversal divergence theorem
\cite{YT} on a foliated Riemannian
manifold.
\begin{theorem}  \cite{YT} 
Let $(M,g_M,\mathcal F)$ be a closed, connected Riemannian manifold with a foliation $\mathcal F$ and a bundle-like metric $g_M$ with respect to $\mathcal F$. Then
\begin{align*}
\int_M \operatorname{div_\nabla}(s) = \int_M g_Q(s,\tau)
\end{align*}
for all $s\in\Gamma Q$. 
\end{theorem}
A differential form $\omega\in \Omega^r(M)$ is {\it basic} if
$i(X)\omega=0$ and $i(X)d\omega=0$ for all $X\in T\mathcal F,$ where $i(X)$ is the interior product.
Let $\Omega_B^r(\mathcal F)$ be the set of all basic r-forms on
$M$. Then $\Omega^*(M)=\Omega_B^*(\mathcal F)\oplus \Omega_B^*(\mathcal F)^\perp$  \cite{LO}. Let $\kappa$ be the mean curvature form of $\mathcal F$, which is given by 
\begin{align*}
\kappa(s)=g_Q(\tau,s)
\end{align*}
 for any $s\in Q$. Then the basic part $\kappa_B$ of the mean curvature form is closed, i.e., $d\kappa_B=0$  \cite{LO}.  
Let $d_B$ be the restriction of $d$ on $\Omega_B(\mathcal F)$ and $\delta_B$  its formal adjoint operator of $d_B$ with respect to the global inner product $\ll\cdot,\cdot\gg$, which is given by
\begin{align*}
\ll \phi,\psi\gg = \int_M \phi\wedge\bar *\psi\wedge\chi_\mathcal F
\end{align*}
for any basic $r$-forms $\phi$ and $\psi$, 
where $\bar *$ is the star operator on $\Omega_B^*(\mathcal F)$ and $\chi_{\mathcal F}$ is the characteristic form of $\mathcal F$  \cite{TO}.  The operator $\delta_B$ is given by
\begin{align}
\delta_B \phi= \Big(\delta_T + i(\kappa_B^\sharp)\Big)\phi,\quad \delta_T\phi = (-1)^{q(r+1)+1}\bar * d_B \bar *\phi.
\end{align}
Note that  the induced connection $\nabla$ on $\Omega_B^* (\mathcal F)$  from the connection $\nabla$ on $Q$ and Riemannian connection $\nabla^M$ on $M$  extends the partial Bott connection, which satisfies $\nabla_X\omega = L_X\omega$ for any $X\in  T\mathcal F$ \cite{KT}. 
Then  the operator $\delta_T$ is given by
\begin{align}
\delta_T\phi = -\sum_{a=1}^q i(e_a)\nabla_{e_a}\phi.
\end{align}
The {\it basic Laplacian} $\Delta_B$ acting on
$\Omega_B^*(\mathcal F)$ is defined by
\begin{equation}
\Delta_B=d_B\delta_B+\delta_B d_B.
\end{equation} 
Then for any basic function $f$, we have
\begin{equation}
\Delta_B f =\delta_B d_B f =-\sum_a \nabla_{e_a}\nabla_{ e_a} f +\kappa_B^\sharp(f).
\end{equation}

\begin{remark}
Note that for any basic form $\omega$, the relation between $\delta_B$ and  the ordinary operator $\delta$ is given by
\begin{align}
\delta\omega = \delta_B\omega + *\gamma(\omega),
\end{align}
where $\gamma(\omega) = \pm \bar *\omega\wedge\varphi_0$ and $\varphi_0=d\chi_{\mathcal F} + \kappa\wedge\chi_{\mathcal F}$ with $\varphi_0\wedge\chi_{\mathcal F}=0$ \cite{TO}. If $\omega\in \Omega_B^r (r=0,1)$, then we easily have
\begin{align*}
\gamma(\omega)=0,
\end{align*} 
which implies that 
\begin{align}
\delta\omega = \delta_B\omega,\quad \Delta^M\omega =\Delta_B\omega,
\end{align}
where $\Delta^M = d \delta + \delta d$ is the ordinary Laplacian.
\end{remark}
For later use, we recall the generalized maximum principle for foliation (\cite{JLK}).
\begin{theorem} (\cite{JLK}) Let $(M,g_M,\mathcal F)$ be a closed, connected Riemannian manifold witha foliation $\mathcal F$ and a bundle-like metric $g_M$.  If $(\Delta_B -\kappa_B^\sharp)f \geq 0$ for 
any basic function $f$, then $f$ is constant. 
\end{theorem}
And we review some theorems for transversal nonisometric conformal field (\cite{JJ}).
\begin{theorem}  (\cite{JJ}) Let $(M,g_M,\mathcal F)$ be a closed, connected Riemannian manifold with a foliation $\mathcal F$ of codimension $q$ and bundle-like metric $g_M$ such that $\delta_B\kappa_B=0$. Assume that the transversal scalar curvature $\sigma^Q$ is nonzero constant. Then for any transversal nonisometric conformal field $Y$ such that $L_Y g_Q = 2f_Y g_Q\  (f_Y\ne 0)$, 
\begin{align*}
(\Delta_B -\kappa_B^\sharp)f_Y = {\sigma^Q\over q-1} f_Y \ {\rm and}\ \int_M f_Y =0.
\end{align*}
\end{theorem}

\section{Tensors $E^Q$ and $Z^Q$}
Let $(M,g_M,\mathcal  F)$ be a Riemannian  manifold of codimension $q$-dimensional foliation $\mathcal F$ and a bundle-like metric $g_M$ with respect to $\mathcal F$. 
 In this section, we give the properties of tensors $E^Q$ and $Z^Q$ on a Riemannian foliation. From (1.3) and (1.4),  we have
 \begin{align*}
 \sum_a Z^Q(s,e_a)e_a = E^Q(s)
 \end{align*}
 for any $s\in \Gamma Q$.  Also, we have the following (\cite{JJ}, \cite{JL}).
 \begin{align}
& {\rm tr}_Q E^Q=0,\quad {\rm div}_\nabla (E^Q) = {q-2\over 2q} \nabla \sigma^Q,\\
&|E^Q|^2 =|{\rm Ric}^Q|^2 -{(\sigma^Q)^2\over q},\quad |Z^Q|^2 =|R^Q|^2 -{2(\sigma^Q)^2\over q(q-1)}.
\end{align}
 Now, we recall the Lie derivatives of tensors along the transversal conformal field.
\begin{lemma}  (\cite{JU4}, \cite{JJ}, \cite{JL}) Let $ Y$
be a transversal conformal  field such that $L_Yg_Q=2f_Y g_Q$.  Then
\begin{align}
&g_Q((L_YR^Q)(e_a,e_b)e_c,e_d)=\delta_b^d\nabla_a f_c -\delta_b^c\nabla_a f_d -\delta_a^d\nabla_b f_c + \delta_a^c\nabla_b f_d,\\
&(L_Y{\rm Ric}^Q) (e_a,e_b) = -(q-2)\nabla_a f_b +(\Delta_B f_Y -\kappa_B^\sharp(f_Y))\delta_a^b,\\
&L_Y \sigma^Q = 2(q-1)(\Delta_B f_Y - \kappa_B^\sharp(f_Y))-2f_Y \sigma^Q,\\
&(L_Y E^Q)(e_a,e_b) = -(q-2)\{\nabla_a f_b + \frac 1q
(\Delta_B f - \kappa_B^\sharp(f))\delta_a^b\},\\
&L_Y|E^Q|^2 = -2(q-2)g_Q(\nabla \nabla f_Y,E^Q)-4f_Y |E^Q|^2,\label{4-8}\\
& L_Y|Z^Q|^2 = -8 g_Q(\nabla\nabla f_Y, E^Q) - 4 f_Y|
Z^Q|^2 .
\end{align}
where $\nabla_a =\nabla_{e_a}$ and $f_a =\nabla_a f_Y$.
\end{lemma}

\begin{lemma}  If a transversal conformal field $Y$ satisfies $L_Y{\rm Ric}^Q = \mu g_Q$ for some basic function $\mu$, then
\begin{align*}
L_Y E^Q =0.
\end{align*}
\end{lemma}
\begin{proof}
 Let $Y$ be the transversal conformal field such that $L_Yg_Q=2f_Y g_Q$.   From (3.4), we have
\begin{align}
-(q-2)\nabla_a f_b +(\Delta_B f_Y -\kappa_B^\sharp(f_Y))\delta_a^b = \mu\delta_a^b.
\end{align}
From (2.4) and (3.9), we have
\begin{align}
\mu = {2(q-1)\over q} (\Delta_B f_Y -\kappa_B^\sharp(f_Y)).
\end{align}
From (3.9) and (3.10), we have
\begin{align*}
-(q-2)\Big\{\nabla_a f_b +\frac1q(\Delta_B f_Y -\kappa_B^\sharp(f_Y))\delta_a^b\Big\}=0.
\end{align*}
 Therefore, the proof follows from  (3.6). 
 \end{proof}
 
 \begin{lemma} If $Y$ is a transversal conformal field, then
\begin{align}
L_Y|E^Q|^2 = 2g_Q(L_YE^Q,E^Q).
\end{align}
\end{lemma}
\begin{proof}
Let $\{e_a\}$ be a local orthonormal basis on $Q$ such that $(\nabla e_a)_x =0$ at a point $x$.  Let $Y$ be  the transversal conformal field $Y$ such that $L_Yg_Q=2f_Y g_Q$. Then  at $x$, we have
\begin{align}
L_Y|E^Q|^2 &= \sum_a L_Y g_Q(E^Q(e_a),E^Q(E_a))\notag\\
&=\sum_a (L_Yg_Q) (E^Q(e_a),E^Q(e_a)) + 2\sum_a g_Q((L_YE^Q)(e_a),E^Q(e_a))\notag\\
& +2 \sum_a g_Q(E^Q(L_Ye_a),E^Q(e_a))\notag\\
&= 2f_Y |E^Q|^2 +2g_Q(L_YE^Q,E^Q) +2\sum_a g_Q(E^Q(L_Ye_a),E^Q(e_a)).
\end{align}
Now, we calculate the last term in the above equation. That is,
\begin{align*}
&\sum_a g_Q(E^Q(L_Ye_a),E^Q(e_a))\\
&=\sum_{a,b} g_Q(E^Q(L_Ye_a),e_b)g_Q(E^Q(e_a),e_b)\\
&=\sum_{a,b} g_Q(E^Q(e_b),L_Ye_a) g_Q(E^Q(e_b),e_a)\\
&=\frac12 \sum_{a,b} L_Y \{g_Q(E^Q(e_b),e_a) g_Q(E^Q(e_b),e_a)\} -2f_Y |E^Q|^2\\
&-\sum_{a} g_Q((L_YE^Q)(e_a),E^Q(e_a)) -\sum_a g_Q(E^Q(L_Ye_a),E^Q(e_a)).
\end{align*}
Hence we have
\begin{align}
2\sum_a g_Q(E^Q(L_Ye_a),E^Q(e_a))= &\frac12 L_Y|E^Q|^2 -2f_Y |E^Q|^2\\
& -g_Q(L_YE^Q,E^Q).\notag
\end{align}
From (3.12) and (3.13),  the proof of (3.11)  follows
. 
\end{proof}

\begin{lemma} Let $Y$ be a transversal conformal field such that $L_Yg_Q =2f_Y g_Q$. Then
\begin{align}
&L_Y|Z^Q|^2 = 2g_Q(L_YZ^Q,Z^Q) -4f_Y |Z^Q|^2 \\
&(q-2)g_Q(L_YZ^Q,Z^Q) = 4 g_Q (L_YE^Q,E^Q) + 8 f_Y |E^Q|^2.
\end{align}
\end{lemma}
\begin{proof}
 Note that   $g_Q(L_YZ^Q,Z^Q) = -4 g_Q(\nabla\nabla f_Y, E^Q)$  \cite{JL}. So (3.14) follows from (3.8). For the proof of (3.15),  from (3.7) and (3.8), 
\begin{equation*}
4 L_Y |E^Q|^2 = (q-2) L_Y |Z^Q|^2 + 4(q-2) f_Y |Z^Q|^2 -16 f_Y |E^Q|^2.
\end{equation*}
Hence from (3.11) and (3.14),  the equation (3.15) is proved. $\Box$

 \end{proof}
From (3.2) and Theorem E,  we have the following.
\begin{proposition} Let $(M,g_M,\mathcal F)$ be a closed, connected Riemannian manifold with a minimal foliation $\mathcal F$ of codimension $q\geq 2$ and a bundle-like metric $g_M$. Assume that  the transversal scalar curvature is nonzero constant and either $|Ric^Q|$  or $|R^Q|$  is constant. If $M$ admits a transversal nonisometric conformal field, then $\mathcal F$ is transversally isometric to the sphere $S^q(1/c),G)$.
\end{proposition} 
\begin{remark}  For the ordinary manifold, Proposition 3.5 has been proved in  \cite{LI} and \cite{HS}, respectively.
\end{remark}

\section{The proofs of  Theorems}

First, we recall the integral formulas for the tensor $E^Q$ and $Z^Q$ 
\begin{proposition}   (\cite{JU4}, \cite{JL})  Let $(M,g_M,\mathcal F)$ be a closed, connected Riemannian  manifold with a foliation $\mathcal F$ of codimension $q$ and a bundle-like metric $g_M$ with respect to $\mathcal F$.    Assume that the transversal scalar curvature  $\sigma^Q$ is nonzero constant.  Then for any  transversal nonisometric conformal field $Y$ such that $L_Yg_Q=2f_Y g_Q\ (f_Y\ne 0)$,  we have 
\begin{align} \label{2-7}
2(q-2)\int_M g_Q(E^Q(\nabla f_Y),\nabla f_Y)&=\int_M\{ 4f_Y^2 |E^Q|^2 + f_Y L_Y|E^Q|^2\}\\
&\ +2(q-2)\int_M g_Q(E^Q(f_Y \nabla f_Y),\kappa_B^\sharp)\notag
\end{align}
and
\begin{align}
\int_M g_Q(E^Q(\nabla f_Y),\nabla f_Y)&={1\over 2}\int_M\{ f_Y^2 |Z^Q|^2 + \frac14f_Y L_Y|Z^Q|^2\}\\
&\ \int_M g_Q({\rm Ric}^Q(f_Y \nabla f_Y),\kappa_B^\sharp)\notag
\end{align}

\end{proposition} 

\noindent
{\it Proof of Theorem 1.}  Let $Y$ be the transversal nonisometric conformal field such that $L_Y g_Q = 2f_Y g_Q$. From Theorem 2.3, we have 
\begin{align}
\int_M f_Y =0. 
\end{align}
Assume that $\mathcal F$ is minimal. Since $L_Y|E^Q|^2=const$ or $L_Y|Z^Q|^2=const$, from (4.3) and Proposition 4.1, we have
\begin{align*}
2(q-2)\int_M g_Q(E^Q(\nabla f_Y),\nabla f_Y)=4\int_M f_Y^2 |E^Q|^2
\end{align*}
or
\begin{align*}
\int_M g_Q(E^Q(\nabla f_Y),\nabla f_Y)={1\over 2}\int_Mf_Y^2 |Z^Q|^2,
\end{align*}
respectively. Hence from Theorem D, the proof is completed.

\begin{lemma}  Let  $Y$ be a transversal conformal field such that $L_Yg_Q = 2f_Y g_Q$.  Then for any basic function $h$, 
\begin{align*}
\int_M hf_Y   = -{1\over q}\int_M L_Yh   + {1\over q}\int_M {\rm div}_\nabla(hY) .
\end{align*}
\end{lemma}
{\bf Proof.}  Let $\omega = Y^b$ be the dual basic 1-form of the transversal conformal form $Y$.  Then
\begin{align*}
\int_M h(\delta_B\omega)   = \int_M g_Q (\omega, d_B h)
=\int_M i(Y) d_B h  = \int_M L_Yh .
\end{align*}
Since $\delta_B = \delta_T + i(\kappa_B^\sharp)$ and $\delta_T\omega  =-{\rm div}_\nabla (Y)=-qf_Y$, we have
\begin{align*}
q\int_M hf_Y  &= -\int_Mh (\delta_T\omega) \\
&=-\int_M h(\delta_B\omega) + \int_M h i(\kappa_B^\sharp) \omega \\
&= -\int_M L_Y h  + \int_M g_Q (h Y,\kappa_B^\sharp) \\
&=-\int_M L_Y h + \int_M {\rm div}_\nabla (hY).
\end{align*}
Last equality in above follows from the transversal divergence theorem (Theorem 2.1).  Therefore, the proof is completed. $\Box$

\bigskip
\noindent{\it Proof of Theorem 2.}
  Let $Y$ be a transversal nonisometric conformal field, i.e., $L_Yg_Q =2f_Y g_Q$. From (2.6),  Lemma 3.4 and Proposition 4.1,  if we put $h=g_Q(L_YE^Q,E^Q)$, then from Lemma 4.2, we have
\begin{align*}
(q-2)\int_M& g_Q(E(\nabla f_Y),\nabla f_Y)\\
&=2\int_M f_Y^2 |E^Q|^2 + \int_M hf_Y + (q-2)\int_M g_Q(E(f_Y\nabla f_Y),\kappa_B^\sharp)\\
&=2\int_M f_Y^2 |E^Q|^2 -{1\over q}\int_M L_Y h+{1\over q}\int_M g_Q(hY,\kappa_B^\sharp)\\
&+(q-2) \int_M g_Q(E^Q(f_Y\nabla f_Y),\kappa_B^\sharp).
\end{align*}
Since $\mathcal  F$ is minimal, we have
\begin{align*}
(q-2)\int_M g_Q(E^Q(\nabla f_Y),\nabla f_Y)=2\int_M f_Y^2 |E^Q|^2 -{1\over q}\int_M L_Y g_Q(L_YE^Q,E^Q).
\end{align*}
Hence by the condition $L_Y g_Q(L_Y E^Q,E^Q)\leq 0$, we have
\begin{align*}
\int_M g_Q(E^Q(\nabla f_Y),\nabla f_Y) \geq 0.
\end{align*}
From Theorem D, the proof of Theorem 2 is completed.

\begin{remark}
Let $\mathcal F$ be minimal.  Then

(1)  From Lemma  3.3,  Theorem 2 yields Theorem A  (2). 

(2)  Theorem 2 is also a generalization of Theorem A (3).  In fact, assume that ${\rm Ric}^Q(\nabla f_Y)={\sigma^Q\over q}\nabla f_Y$, that is, $E^Q(\nabla f_Y)=0$. By differentiation, we have
\begin{align}
(\nabla_{e_a}E^Q) (\nabla f_Y) + E^Q(\nabla_a\nabla f_Y)=0.
\end{align}
From (4.4),  we have
\begin{align}
0&=\sum_a g_Q( (\nabla_{e_a} E^Q)(\nabla f_Y) +E^Q(\nabla_a \nabla f_Y),e_a) \notag\\
&= g_Q(\nabla f_Y,{\rm div}_\nabla(E^Q)) + \sum_a g_Q (E^Q(\nabla_a \nabla f_Y),e_a)\notag\\
&= \sum_a g_Q (\nabla_a\nabla f_Y, E^Q(e_a)).
\end{align}
From (3.1),  ${\rm div}_\nabla E^Q =0$ and so the last equality in the above  follows.  
Hence from (3.6) and (4.5), we have
\begin{align*}
g_Q(L_YE^Q,E^Q) &= \sum_a g_Q ((L_Y E^Q)(e_a), E^Q(e_a))\\
 &= -(q-2) \sum_a g_Q(\nabla_a \nabla f_Y, E^Q(e_a)) - {q-2\over q} (\Delta_B f_Y) \sum_a g_Q (e_a,E^Q(e_a))\\
 &=-(q-2) \sum_a g_Q(\nabla_a \nabla f_Y, E^Q(e_a)) - {q-2\over q} (\Delta_B f_Y) {\rm tr}_Q E^Q\\
 &=0.
 \end{align*}
 The last equality follows from ${\rm tr}_Q E^Q =0$.   Hence Theorem A (3) implies that $g_Q(L_Y E^Q,E^Q)=0$.  
\end{remark}
\bigskip
\noindent{\it Proof of Theorem 3.}  
 Let $Y$ be a transversal conformal field  such that $L_Yg_Q=2f_Y g_Q$ and  $Y= K + \nabla h$, where $K$ is a transversal Killing field and $h$ is a basic function. Then
\begin{align*}
g_Q(\nabla_XY,Z) + g_Q (\nabla_ZY,X) =2f_Y g_Q (X,Z)
\end{align*}
for any normal vector field $X,Z \in \Gamma Q$.  On the other hand, since  the transversal scalar curvature $\sigma^Q$ is constant, from Theorem 2.4, we have
\begin{align}
(\Delta_B -\kappa_B^\sharp)f_Y = {\sigma^Q\over q-1}f_Y.
\end{align}
Since $Y=K+\nabla h$, we have  $L_Y g_Q=L_{\nabla h} g_Q=2f_Y g_Q$. That is,
\begin{align}
g_Q(\nabla_X \nabla h,Z) + g_Q (\nabla_Z \nabla h,X) = 2f_Y g_Q(X,Z).
\end{align}
On the other hand,  $(\nabla \nabla h)(X,Z) =g_Q (\nabla_X \nabla h,Z)$ is symmetric. Therefore,  from (4.7)
\begin{align}
 (\nabla\nabla h)(X,Z) = f_Yg_Q(X,Z).
 \end{align}
 Hence from (2.4) and (4.8), we have
 \begin{align}
 (\Delta_B -\kappa_B^\sharp)h =-qf_Y.
 \end{align}
 From (4.6) and (4.9), we get
 \begin{align}
 (\Delta_B -\kappa_B^\sharp)\Big( f_Y + {\sigma^Q\over q(q-1)}h\Big)=0.
 \end{align}
 By the generalized maximum principle (Theorem 2.3), we have 
 \begin{align*}
 f_Y + {\sigma^Q\over q(q-1)}h =const,
 \end{align*}
 which implies 
 \begin{align}
 \nabla\nabla f_Y + {\sigma^Q\over q(q-1)}\nabla\nabla h =0.
 \end{align}
 From (4.8) and (4.11), we have
 \begin{align*}
 \nabla\nabla f_Y  = -{\sigma^Q\over q(q-1)} f_Y.
 \end{align*}
 By the generalized Obata theorem  \cite{JLK},  $\mathcal F$ is transversally isometric to the sphere $(S^q(1/c),G)$, where $c^2={\sigma^Q\over q(q-1)}$. 
 
 \begin{remark}
   {\rm Theorem 3 is a generalization of  Theorem A (1).}
 \end{remark}


\begin{thebibliography}{[99]}

\bibitem{LO}\label{LO} J. A. Alvarez L\'{o}pez, \emph{The basic component of
the mean curvature of Riemannian foliations}, Ann. Global Anal.
Geom. 10 (1992), 179-194.

\bibitem{GO} S. I. Goldberg, \emph{Manifolds admitting a one-parameter group of conformal transformations}, Michigan Math. J. 15 (1968), 339-344.

\bibitem{HS} C. C. Hsiung, \emph{On the group of conformal transformations of a compact Riemannian manifold}, Proc. Nat. Acad. Sci. U.S.A. 56 (1965), 1509-1513.

\bibitem{JU4} S. D. Jung, \emph{Riemannian foliations admitting transversal conformal fields II}, Geom. Dedicata 175 (2015), 257-266.

\bibitem{JJ} M. J. Jung and S. D. Jung, \emph{Riemannian foliations admitting transversal conformal fields}, Geom. Dedicata 133 (2008), 155-168.

\bibitem{JL} S. D. Jung and K. R. Lee, \emph{The properties of Riemannian foliations admitting transversal conformal fields},  Bull. Korean Math. Soc. 55 (2018), 1273-1283.

\bibitem{JLK} S. D. Jung, K. R. Lee and K. Richardson, \emph{Generalized Obata theorem and its applications on foliations}, J. Math. Anal. Appl. 376 (2011), 129-135.

\bibitem{KT1} F. W. Kamber and Ph. Tondeur, \emph{Harmonic
foliations}, Proc. National Science Foundation Conference on
Harmonic Maps, Tulane, Dec. 1980, Lecture Notes in Math. 949,
Springer-Verlag, New-York, 1982, 87-121.

\bibitem{KT2}  F. W. Kamber and Ph. Tondeur, \emph{Infinitesimal
automorphisms and second variation of the energy for harmonic
foliations}, Tohoku Math. J. 34 (1982), 525-538.

\bibitem{KT}  F. W. Kamber and Ph. Tondeur, \emph{De Rham-Hodge theory for Riemannian foliations}, Math. Ann. 277 (1987), 415-431.

\bibitem{LR}  J. Lee and K. Richardson, \emph{Lichnerowicz and Obata theorems for
foliations}, Pacific J. Math. 206 (2002), 339-357.

\bibitem{LI}  A. Lichnerowicz,\emph{Sur les transformations conformes d'une vari\'et\'e riemannienne compacte}, C. R. Acad. Sci. Paris 259 (1964), 697-700.

\bibitem{PP}  H. K. Pak and J. H. Park, \emph{A note on generalized Lichnerowica-Obata theorems for Riemannian foliations}, Bull. Korean Math. Soc. 48 (2011), 769-777.

\bibitem{PY} J. S. Pak and S. Yorozu, \emph{Transverse fields on foliated Riemannian manifolds}, J.
Korean Math. Soc. 25 (1988), 83-92.

\bibitem{PY1} J. H. Park and S. Yorozu, \emph{Transversal conformal fields of foliations}, Nihonkai
Math. J. 4 (1993), 73-85.

\bibitem{TO}
Ph. Tondeur, {Geometry of foliations}, Birkh\"auser-Verlag, Basel; Boston: Berlin, 1997.

\bibitem{TT} Ph. Tondeur and G. Toth, \emph{On transversal
infinitesimal automorphisms for harmonic foliations}, Geom.
Dedicata 24 (1987), 229-236.

\bibitem{YA-1} K. Yano, \emph{Concircular geometry I, Concircular transformations}, Proc. Imp. Acad. Tokyo 16 (1940), 195-200.

\bibitem{YA1} K. Yano, { Integral formulas in Riemannian geometry}, Marcel Dekker Inc., 1970.


\bibitem{YA} K. Yano, \emph{On Riemannian manifolds with constant scalar curvature admitting a conformal transformation group}, Proc. Nat. Acad. Sci. U.S. 55 (1966), 472-476.

\bibitem{YA2} K. Yano, \emph{Riemannian manifolds admitting a conformal transformation group}, Proc. Nat. Acad. Sci. U.S.A. 62 (1969), 314-319.

\bibitem{YS} K. Yano and S. Sawaki, \emph{Riemannioan manifolds admitting a  conformal transformation group}, J. Differential Geometry 2 (1968), 161-184.

\bibitem{YT} S. Yorozu and T. Tanemura, \emph{Green's theorem on
a foliated Riemannian manifold and its applications}, Acta Math.
Hungar. 56 (1990), 239-245.

\end{thebibliography}
\end{document}